\documentclass[a4paper]{amsart}
\usepackage{amsmath,amsthm}
\usepackage{amssymb}
\usepackage{microtype}
\usepackage[all]{xy}
\usepackage{color}
\usepackage{combelow}
\usepackage{enumerate}
\newtheorem{theorem}{Theorem}
\newtheorem*{claim*}{Claim}
\newtheorem*{splitting condition}{Splitting Condition}
\newtheorem*{properness condition}{Properness Condition}
\newtheorem*{theorema}{Theorem A}
\newtheorem*{theoremb}{Theorem B}
\newtheorem*{theorem1}{Theorem 1}
\newtheorem*{theorem2}{Theorem 2}
\newtheorem*{theorem3}{Theorem 3}
\newtheorem*{illustration1}{Illustration 1}
\newtheorem*{illustration2}{Illustration 2}
\newtheorem*{illustration3}{Illustration 3}

\newtheorem{corollary}{Corollary}
\newtheorem{definition}{Definition}

\newtheorem{lemma}{Lemma}
\newtheorem{proposition}{Proposition}
\newtheorem{remark}{Remark}

\newcommand{\diffto}{\xrightarrow{\raisebox{-0.2 em}[0pt][0pt]{\smash{\ensuremath{\sim}}}}}

\newcommand{\inc}{\hookrightarrow}
\newcommand{\rmap}{\longrightarrow}
\newcommand{\lmap}{\longleftarrow}

\newcommand{\lb}{\lbrace}
\newcommand{\rb}{\rbrace}

\newcommand{\V}{\mathcal{V}}
\newcommand{\G}{\mathcal{G}}
\newcommand{\rr}{\rightrightarrows}
\newcommand{\pr}{\operatorname{pr}}
\newcommand{\X}{\mathfrak{X}}

\begin{document}
\title[Normal forms for Poisson Maps and Symplectic Groupoids]{Normal forms for Poisson Maps\\and Symplectic Groupoids\\around Poisson Transversals}
\author{Pedro Frejlich}
\address{Depart.\ of Math., Utrecht University, 3508 TA, Utrecht, The Netherlands}
\email{frejlich.math@gmail.com}
\author{Ioan M\u{a}rcu\cb{t}}
\address{Radboud University Nijmegen, IMAPP, 6500 GL, Nijmegen, The Netherlands}
\email{i.marcut@math.ru.nl}
\begin{abstract}
Poisson transversals are those submanifolds in a Poisson manifold which intersect all symplectic leaves transversally and symplectically. In a previous note \cite{PT1} we proved a normal form theorem around such submanifolds. In this communication, we promote that result to a normal form theorem for Poisson maps around Poisson transversals. A Poisson map pulls a Poisson transversal back to a Poisson transversal, and our first main result states that simultaneous normal forms exist around such transversals, for which the Poisson map becomes transversally linear, and intertwines the normal form data of the transversals.

Our second main result concerns symplectic integrations. We prove that a neighborhood of a Poisson transversal is integrable exactly when the Poisson transversal itself is integrable, and in that case we prove a normal form theorem for the symplectic groupoid around its restriction to the Poisson transversal, which puts all its structure maps in normal form.

We conclude the paper by illustrating our results with examples arising from Lie algebras.
\end{abstract}
\maketitle
\tableofcontents
\setcounter{tocdepth}{1}

\section*{Introduction}

Poisson transversals are special submanifolds which play in Poisson geometry a role akin to that of symplectic submanifolds in Symplectic Geometry, and complete transversals in Foliation Theory. A Poisson transversal of a Poisson manifold $(M,\pi)$ is an embedded submanifold $X\subset M$ which intersects all symplectic leaves transversally and symplectically. These submanifolds lie at the heart of Poisson geometry, silently underpinning many important arguments and constructions.

In our previous note \cite{PT1}, we described a normal form theorem around a Poisson transversal $(X,\pi_X)$ in $(M,\pi)$, which depends only on the restriction of $\pi$ to $T^{\ast}M|_X$. Choosing a Poisson spray $\mathcal{V}$ for $\pi$, the corresponding exponential map induces the Poisson isomorphism around $X$ which puts the structure in normal form:
\begin{align}\label{eq : normal form intro}
 \exp_{\V}:(N^{\ast}X,\pi_X^{\omega_{\mathcal{V}}}) \inc (M,\pi),
\end{align}
where $\pi_X^{\omega_{\mathcal{V}}}$ is the Poisson structure corresponding to the Dirac structure $p^*(L_{\pi_X})^{\omega_{\mathcal{V}}}$ obtained as follows: by first pulling back the Dirac structure $L_{\pi_X}$  corresponding to $\pi_X$ to $N^*X$ via the map $p$, and then gauge-transforming by a certain closed two-form $\omega_{\V}$ on $N^*X$ which is symplectic on the fibers of $p$. Actually, all these objects ($\exp_{\V}$, $\omega_{\mathcal{V}}$ and $\pi_X^{\omega_{\mathcal{V}}}$) are only defined on a small open neighborhood of $X \subset N^*X$, but we omit this technicality from the notation. The procedure in \cite{PT1} for constructing normal forms as in (\ref{eq : normal form intro}) depends only on the choice of $\mathcal{V}$, and has the added benefit of allowing simultaneous normal forms for all Poisson transversals in $(M,\pi)$.

In this communication, we continue our analysis of local properties around Poisson transversals with normal form results for Poisson maps and symplectic groupoids.

That Poisson transversals behave functorially with respect to Poisson maps has already been pointed out in \cite{PT1}: a Poisson map pulls back Poisson transversals to Poisson transversals, and in fact, it pulls back the corresponding infinitesimal data pertaining to their normal forms. We prove that the two Poisson structures and the Poisson map can be put in normal form simultaneously:

\begin{theorem1}[Normal form for Poisson maps]\indent
Let $\varphi : (M_0,\pi_0) \to (M_1,\pi_1)$ be a Poisson map, and $X_1 \subset M_1$ be a Poisson transversal. Then $\varphi$ is transverse to $X_1$, $X_0:=\varphi^{-1}X_1$ is a Poisson transversal in $(M_0,\pi_0)$, $\varphi|_{X_0}:(X_0,\pi_{X_0}) \to (X_1,\pi_{X_1})$ is a Poisson map, and there exist Poisson sprays $\V_i$ with exponential maps $\mathrm{exp}_{\V_i}:(N^{\ast}X_i,\pi_{X_i}^{\omega_{\V_i}}) \inc (M_i,\pi_i)$ which fit into the commutative diagram of Poisson maps:
\[\xymatrix{
 (M_0,\pi_0) \ar[r]^{\varphi} & (M_1,\pi_1) \\
 (N^{\ast}X_0,\pi_{X_0}^{\omega_{\V_0}}) \ar[u]^{\exp_{\V_0}} \ar[r]_F & (N^{\ast}X_1,\pi_{X_1}^{\omega_{\V_1}})  \ar[u]_{\exp_{\V_1}}}\]
where $F$ is the vector bundle map:
\begin{align*}
F_{x}:=(\varphi^*|_{N^*_{\varphi(x)}X_1})^{-1} : N^{\ast}_xX_0 \rmap N^{\ast}_{\varphi(x)}X_1,
\end{align*}
and, moreover, $F$ satisfies $F^*(\omega_{\V_1})=\omega_{\V_0}$.
\end{theorem1}

Two comments are now in order. First, there is a great dearth of normal form theorems for Poisson maps in the literature: one can find some scant precedents in the normal form for moment maps on symplectic manifolds of \cite{GuillSternb84,Marle85}, or in some normal forms belonging to the theory of integrable systems (e.g. \cite{Nehorosev72,Duistermaat,Eli90}). Second, that such a simple -- and somewhat unexpected -- normal form can be proved is further testament of the central role played by Poisson transversals in Poisson geometry, and owes greatly to the canonicity of the methods they grant.

Next, we move to symplectic groupoids. As a general principle, which follows from the normal form theorem, Poisson transversals encode all the geometry of a neighborhood in the ambient manifold, and 'transverse properties' should hold for the transversal if and only if they hold true around it. We show that integrability by a symplectic groupoid is one such transverse property:
\begin{theorem2}[Integrability as a transverse property]
A Poisson transversal is integrable if and only if it has an integrable open neighborhood.
\end{theorem2}

In fact, we show much more:
\begin{theorem3}[Normal form for symplectic groupoids]

Let $(X,\pi_X)$  be a Poisson transversal in $(M,\pi)$, and consider a tubular neighborhood $M \supset E \stackrel{p}{\to} X$ in which the Poisson structure is in normal form, i.e.\ $\pi|_E=\pi_X^{\sigma}$. If $(X,\pi_X)$ is integrable by a symplectic groupoid $(\mathcal{G}_X,\omega_X) \rr (X,\pi_X)$, then:

\begin{enumerate}[a)]
 \item A symplectic groupoid integrating $\pi_X^{\sigma}$ is $(\mathcal{G}_X^E,\omega_E) \rr (E,\pi_X^{\sigma})$, where:
\[
 \mathcal{G}_X^E:=\mathcal{G}_X \times_{\mathcal{P}(X)}\mathcal{P}(E), \ \ \ \ \omega_E:=\mathbf{p}^{\ast}(\omega_{X})+\mathbf{s}^{\ast}(\sigma)-\mathbf{t}^{\ast}(\sigma).\]
 Here $\mathcal{P}(\mathcal{M}) \rr \mathcal{M}$ stands for the pair groupoid of a manifold $\mathcal{M}$, and $\mathbf{p}:\mathcal{G}_X^E \to \mathcal{G}_X$ stands for the canonical groupoid map.
 \item The restriction to $E$ of any symplectic groupoid $(\mathcal{G},\omega_{\mathcal{G}}) \rr (M,\pi)$ integrating $\pi$ is isomorphic to the model $(\mathcal{G}_X^E,\omega_E)$ corresponding to $\mathcal{G}_X:=\mathcal{G}|_X$, $\omega_X:=\omega_{\mathcal{G}}|_{\mathcal{G}_X}$.
\end{enumerate}
\end{theorem3}

We conclude the paper by illustrating our results in the setting of linear Poisson structures, i.e., Lie algebras. While in this linear setting the conclusions of Theorems 1-3 are well-known, we strive to show how a Poisson-transversal perspective can shed new light on even these classical results.

\medskip

\subsection*{Acknowledgements}
We would like to thank Marius Crainic for the many good conversations, and for his insight that Lemma \ref{lem : Poisson pulls back Poisson transversals} should be a shadow of a normal form theorem for Poisson maps. We would also like to thank David Mart\'{i}nez-Torres and Rui Loja Fernandes for useful discussions.

The first author was supported by the NWO Vrije Competitie project ``Flexibility and Rigidity of Geometric Structures'' no.\ 613.001.101 and the second was partially supported by the NSF grant DMS 14-05671.

\section{Preliminaries on Poisson transversals}

Recall from \cite{PT1} that an embedded submanifold $X \subset M$ in a Poisson manifold $(M,\pi)$ is said to be a {\bf Poisson transversal} if it induces a splitting:
\begin{equation}\label{eq : splitting}
TX \oplus NX=TM|_X,
\end{equation}
where $NX:=\pi^{\sharp}(N^{\ast} X)\subset TM|_{X}$ will be called the {\bf embedded normal bundle}. As explained in \cite{PT1}, the restriction $\pi|_{X}$ of $\pi$ to $T^{\ast}M|_X$ decomposes as:
\[\pi|_{X}=\pi_{X}+w_{X},\]
where $\pi_{X}\in \Gamma(\bigwedge^2TX)$ is a Poisson structure and $w_X \in \Gamma(\bigwedge^2N X)$ is a non-degenerate bivector. The main result of \cite{PT1} is that pair $(\pi_X,w_X)$ encodes the structure of $\pi$ around $X$. To explain this, recall:

\begin{definition}
Let $(M,\pi)$ be a Poisson manifold. A vector field $\V \in \X(T^{\ast}M)$ is a {\bf spray} for $\pi$ if:
\begin{enumerate}
 \item $m_t^{\ast}(\V)=t\V$, for all $t>0$;
 \item $\pr_{\ast}\V(\xi)=\pi^{\sharp}(\xi)$, for all $\xi\in T^{\ast}M$,
\end{enumerate}
where $m_t:T^{\ast}M \to T^{\ast}M$ denotes the map of scalar multiplication by $t$.
\end{definition}

The following result played a crucial role in the proof of the normal form theorem in \cite{PT1} :

\begin{theorema}\cite{CrMar11}
Let $\pi$ be Poisson and denote by $\phi_t$ the time-$t$ (local) flow of a spray $\V$ for $\pi$. Then there is an open $\Sigma_{\V} \subset T^{\ast}M$ around $M$ with the property that:
\begin{enumerate}
 \item $\phi$ is defined on $\Sigma_{\V} \times [0,1]$;
 \item The closed two-form $\Omega_{\V}:=\int_0^1\phi_t^{\ast}\omega_{\mathrm{can}}dt$ is symplectic on $\Sigma_{\V}$;
 \item The submersions
 \[(M,\pi)\stackrel{\pr}{\lmap}(\Sigma_{\V},\Omega_{\V})\stackrel{\exp_{\V}}{\rmap}(M,-\pi)\]give a full dual pair, where $\exp_{\V}:=\pr\circ\phi_1$.
\end{enumerate}
\end{theorema}

Let $X\subset (M,\pi)$ be a Poisson transversal with associated pair $(\pi_X,w_X)$. We denote by $\Upsilon(w_X)$ the space of all closed two-forms $\sigma \in \Omega^2(N^{\ast}X)$ which along $X$ satisfy $\sigma|_X=-w_X\in \Gamma(\bigwedge^2 NX)$, where we identify $\bigwedge^2 NX$ with the space of vertical two-forms in $\bigwedge^2 T^*(N^*X)|_X$. To each $\sigma \in \Upsilon(w_X)$ there corresponds a local model of $\pi$ around $X$, which, in Dirac-geometric terms, is described as the Poisson structure $\pi_X^{\sigma}$ corresponding to the Dirac structure $\pr^*(L_{\pi_X})^{\sigma}$. As shown in \cite{PT1}, $\pi_X^{\sigma}$ is defined in a neighborhood of $X$ in $N^*X$, and for any other $\sigma' \in \Upsilon(w_X)$, $\pi_X^{\sigma}$ and $\pi_X^{\sigma'}$ are Poisson diffeomorphic around $X$, by a diffeomorphism that fixes $X$ to first order.

\begin{theoremb}\cite{PT1}
In the notation of Theorem A, the two-form $\omega_{\V}:=-\Omega_{\V}|_{N^{\ast}X}$ belongs to $\Upsilon(w_X)$, and the exponential map yields a Poisson embedding around $X$,
 \[
 \exp_{\V}:(N^{\ast}X,\pi_X^{\omega_{\V}}) \inc (M,\pi).
 \]
\end{theoremb}

\begin{remark}\label{rem : locally defined}
In Theorem B, $\exp_{\V}$, and $\omega_{\V}$ are defined only on small enough neighborhoods of $X$ in $N^*X$, but we still write $\exp_{\V}:N^{\ast}X\to M$, and $\omega_{\V}\in \Upsilon(w_X)$. This convention will be used throughout Section \ref{sec : normal form for poisson maps around transversals}, also for other maps and tensors, as it simplifies notation considerably.
\end{remark}

\section{Normal form for Poisson maps}\label{sec : normal form for poisson maps around transversals}

The result below is a the first indication for a normal form theorem for Poisson maps should hold around Poisson transversals; we refer the reader to \cite{PT1} for a proof:

\begin{lemma}\label{lem : Poisson pulls back Poisson transversals}
Let $\varphi:(M_0,\pi_0)\to (M_1,\pi_1)$ be a Poisson map and $X_1\subset M_1$ be a Poisson transversal. Then:
\begin{enumerate}
 \item $\varphi$ is transverse to $X_1$;
 \item $X_0:=\varphi^{-1}(X_1)$ is also a Poisson transversal;
 \item $\varphi$ restricts to a Poisson map $\varphi|_{X_0}:(X_0,\pi_{X_0}) \to (X_1,\pi_{X_1})$;
 \item The differential of $\varphi$ along $X_0$ restricts to a fibrewise linear isomorphism between embedded normal bundles $\varphi_*|_{NX_0}:NX_0\to NX_1$;
 \item The map $F:N^*X_0\to N^*X_1$, $F(\xi)=(\varphi^*)^{-1}(\xi)$, $\xi\in N^*X_0$ is a fibrewise linear symplectomorphism between the symplectic vector bundles
 \[F:(N^*X_0,w_{X_0})\to (N^*X_1,w_{X_1}).\]
\end{enumerate}
\end{lemma}

We are ready to state the main result of this section. Consider now the same setting as in Lemma \ref{lem : Poisson pulls back Poisson transversals}:
$\varphi:(M_0,\pi_0) \to (M_1,\pi_1)$ is a Poisson map, $X_1 \subset (M_1,\pi_1)$ is a Poisson transversal, and consider the fibrewise symplectomorphism
\begin{equation}\label{EQ : map F}
F:(N^*X_0,w_{X_0})\to (N^*X_1, w_{X_1}).
\end{equation}

\begin{theorem}[Normal form for Poisson maps]\label{thm : Normal Form for Poisson maps}\indent
There are sprays $\V_0$ for $\pi_0$, and $\V_1$ for $\pi_1$, so that under the induced exponentials $\exp_{\V_i}:(N^{\ast}X_i,\pi_{X_i}^{\omega_{\V_i}}) \inc (M_i,\pi_i)$, the map $\varphi$ corresponds to the bundle map $F$, and $F^*(\omega_{\V_1})=\omega_{\V_0}$. In particular, we have a commutative diagram of Poisson maps:
\[\xymatrix{
 (M_0,\pi_0) \ar[r]^{\varphi} & (M_1,\pi_1) \\
 (N^{\ast}X_0,\pi_{X_0}^{\omega_{\V_0}}) \ar[u]^{\exp_{\V_0}} \ar[r]^F & (N^{\ast}X_1,\pi_{X_1}^{\omega_{\V_1}})  \ar[u]_{\exp_{\V_1}}}\]
\end{theorem}

In other words, the theorem allows us to bring simultaneously both Poisson structures in normal form so that the Poisson map becomes linear in the normal directions. This specializes to the normal form theorem of \cite{PT1} by taking $M_0=M_1$, $X_0=X_1$ and $\varphi=\operatorname{id}$. Remark \ref{rem : locally defined} applies also here: the result is only local around $X_0$ and $X_1$, as the exponential maps $\exp_{\V_i}$ are defined only around $X_i$. Moreover, if $X_i$ are not closed submanifolds, then we can only guarantee that the sprays $\V_i$ be defined around $X_i$.

\begin{proof}[Proof of Theorem \ref{thm : Normal Form for Poisson maps}] We split the proof into four steps:\\

\noindent\emph{Step 1: Extending the map $F$ around $X_0$ and $X_1$.} Let $U_1\subset M_1$ be an open set containing $X_1$ on which there is a vector subbundle $A_1\subset T^*U_1$ extending $N^*X_1$, i.e.\ $A_1|_{X_1}=N^*X_1$. Consider the vector bundle map between vector bundles over $\varphi^{-1}(U_1)\subset M_0$:
\[\varphi^*:\varphi^{-1}(U_1)\times_{U_1}A_1\to T^*\varphi^{-1}(U_1).\]
By Lemma \ref{lem : Poisson pulls back Poisson transversals}, this map is fibrewise injective along $X_0$. Therefore, it is injective on an open neighborhood $U_0$ of $X_0$ in $\varphi^{-1}(U_1)$. Let \[A_0:=\varphi^*(U_0\times_{U_1}A_1)\subset T^*U_0.\]
Then $A_0|_{X_0}=N^*X_0$. Clearly $\varphi^*:U_0\times_{U_1}A_1\diffto A_0$ is a vector bundle isomorphism. The inverse of this map, composed with the second projection, gives the vector bundle map
\[
\xymatrix{
 A_0 \ar[r]^{\widetilde{F}}\ar[d] & A_1\ar[d] \\
 U_0\ar[r]^{\varphi} & U_1}
 \]
which is a fibrewise isomorphism and extends the map $F$ in (\ref{EQ : map F}).\\

\noindent\emph{Step 2: Constructing $\widetilde{F}$-related sprays on $A_0$ and $A_1$.} Consider a spray on $A_1$, i.e.\ a quadratic vector field $\V_1\in\X(A_1)$ so that $\pr_*(\V_1(\xi))=\pi_1^{\sharp}(\xi)$, for $\xi\in A_1$. Such a spray can be easily constructed by choosing a linear connection on $A_1$ and letting $\V_1(\xi)$ be the horizontal lift of $\pi_1^{\sharp}(\xi)$.

We identify $TA_0\cong TU_0\times_{TU_1}TA_1$ by means of the isomorphism $\varphi^*:U_0\times_{U_1}A_1\diffto A_0$; note that, under this identification, the differential of $\widetilde{F}$ becomes the second projection. This shows that there is a unique spray $\V_0$ on $A_0$ that is $\widetilde{F}$-related to $\V_1$: it is defined by the pair $(\pi_0^{\sharp}(\xi),\V_1((\varphi^*)^{-1}(\xi)))\in TU_0\times_{TU_1}TA_1$, for $\xi\in A_0$. That this is indeed an element of $TU_0\times_{TU_1}TA_1$ follows from the fact that $\varphi$ is Poisson, and that $\V_1$ is a spray, namely: \[\varphi_*(\pi_0^{\sharp}(\xi))=\pi_1^{\sharp}((\varphi^*)^{-1}(\xi))=\pr_*(\V_1((\varphi^*)^{-1}(\xi))).\]

\noindent\emph{Step 3: Extending the sprays.} Each spray $\V_i$ can be extended to a spray $\widetilde{\V}_i$ on $T^*U_i$ which is tangent to $A_i$. Choose subbundles $C_i \subset T^*U_i$ complementary to $A_i$; this allows us to describe the tangent bundle to $T^*U_i=A_i\times_{U_i} C_i$ as $T(T^*U_i)=TA_i\times_{TU_i} TC_i$. Choose also linear connections on $A_i$ and $C_i$, the horizonal lifts of which we denote by $h^A$ and $h^C$, respectively. For $\xi=(a,c)\in T^*U_i$, define $\widetilde{\V}_i(\xi):=(\V_i(a)+h^A(\pi_i^{\sharp}(c)), h^C(\pi_i^{\sharp}(\xi)))$. It is easy to see that $\widetilde{\V}_i$ defines a quadratic vector field on $T^*U_i$, and that it is a spray. Since the connections are linear, the canonical inclusion $TU_i\subset TC_i|_{U_i}$ is realized by $h^C$. Also, $TA_i\subset T(T^*U_i)|_{A_i}$ corresponds to $TA_i\times_{TU_i} TU_i \subset (TA_i\times_{TU_i} TC_i)|_{A_i}$. Thus, for $a\in A_i$, we have that $\widetilde{\V}_i(a)=(\V_i(a), \pi_i^{\sharp}(a))\in TA_i$; hence $\widetilde{\V}
_i$ is tangent to $A_i$
and extends $\V_i$.

If $X_i$ is a closed submanifold of $M_i$, then $\widetilde{\V}_i$ can be extended to the entire $T^*M_i$.

To simplify notation we will denote $\widetilde{\V}_i$ also by $\V_i$.\\

\noindent\emph{Step 4: Commutativity of the diagrams.} Let $\Phi_{\V_i}^t$ denote the time-$t$ local flow of $\V_i$. Since $\V_i$ is tangent to $A_i$, and $\widetilde{F}_{\ast}\V_0=\V_1$, on $A_0$ we have that $\widetilde{F}\circ \Phi_{\V_0}^t=\Phi_{\V_1}^t\circ \widetilde{F}$. Since $\widetilde{F}$ extends $F$, we obtain the following commutative diagram:
\[
\xymatrix{
 N^*X_0 \ar[r]^{\Phi^1_{\V_0}}\ar[d]_{F} & A_0\ar[d]^{\widetilde{F}}\ar[r]^{\pr} & U_0\ar[d]^{\varphi}\\
 N^*X_1 \ar[r]_{\Phi^1_{\V_1}} & A_1\ar[r]_{\pr} & U_1\\
}
 \]
which implies the equality $\varphi\circ \exp_{\V_0}=\exp_{\V_1}\circ F$ from the statement.\\

\noindent\emph{Step 5: Compatibility of the two-forms.}
As in Theorems A and B, we denote by $\Omega_{\V_i}:=\int_0^1(\Phi^t_{\V_i})^{\ast}\omega_{\mathrm{can}}dt$ and $\omega_{\V_i}:=-\Omega_{\V_i}|_{N^*X_i}$. By Theorem B, the exponentials $\exp_{\V_i}:(N^*X_i,\pi_{X_i}^{\omega_{\V_i}})\inc (M_i,\pi_i)$ are Poisson diffeomorphisms around $X_i$.
Hence, also $F:(N^*X_0,\pi_{X_0}^{\omega_{\V_0}})\to (N^*X_1,\pi_{X_1}^{\omega_{\V_1}})$ is a Poisson map in a neighborhood of $X_0$. This does not directly imply that $F^*(\omega_{\V_1})=\omega_{\V_0}$, and this is what we prove next.

Recall that the tautological one-form $\lambda_{\mathrm{can}}\in \Omega^1(T^*M_i)$ is defined by $\lambda_{\mathrm{can},\xi}(v):=\langle\xi,\pr_*(v)\rangle$, for $\xi\in T^*M_i$ and $v\in T_{\xi}(T^*M_i)$. We show now that $\widetilde{F}$ satisfies: $\widetilde{F}^*(\lambda_{\mathrm{can}}|_{A_1})=\lambda_{\mathrm{can}}|_{A_0}$. For $\xi\in A_0$ and $v\in T_{\xi}A_0$, we have:
\begin{align*}
(\widetilde{F}^*\lambda_{\mathrm{can}})_{\xi}(v)&= \langle \widetilde{F}(\xi),\pr_*(\widetilde{F}_*(v))\rangle=\langle (\varphi^*)^{-1}(\xi),(\pr\circ\widetilde{F})_*(v)\rangle=\\
&=\langle (\varphi^*)^{-1}(\xi),(\varphi\circ \pr)_*(v)\rangle=\langle (\varphi^*)^{-1}(\xi),\varphi_*(\pr_*(v))\rangle=\\
&=\langle \xi,\pr_*(v)\rangle=\lambda_{\mathrm{can},\xi}(v).
\end{align*}
This implies that $\widetilde{F}^*(\omega_{\mathrm{can}}|_{A_1})=\omega_{\mathrm{can}}|_{A_0}$. Using that $\widetilde{F}$ intertwines the flows of the sprays, and that these flows preserve the submanifolds $A_0$, $A_1$, we obtain:
\begin{align*}
(\Phi_{\V_0}^{t*}\omega_{\mathrm{can}})|_{A_0}&=\Phi_{\V_0}^{t*}(\omega_{\mathrm{can}}|_{A_0})=\Phi_{\V_0}^{t*}\circ \widetilde{F}^*(\omega_{\mathrm{can}}|_{A_1})=\\
&=\widetilde{F}^*\circ \Phi_{\V_1}^{t*}(\omega_{\mathrm{can}}|_{A_1})=\widetilde{F}^*(\Phi_{\V_1}^{t*}(\omega_{\mathrm{can}})|_{A_1}).
\end{align*}
Averaging this equality for $t\in [0,1]$, in neighborhood of $X_0$ yields $\widetilde{F}^*(\Omega_{\V_1}|_{A_1})=\Omega_{\V_0}|_{A_0}$. Restricting to $N^*X_0$, we obtain the conclusion: $F^*(\omega_{\V_1})=\omega_{\V_0}$.
\end{proof}

\section{Integrability}

Symplectic groupoids are the natural objects integrating Poisson manifolds. In this section we discuss the relation between integrability of a Poisson manifold and integrability of one of its transversals. For integrable Poisson manifolds, we give a normal form theorem for the symplectic groupoid around its restriction to a Poisson transversal.

\subsection*{Symplectic groupoids}

We recollect here a few facts about symplectic groupoids and integrability of Poisson manifolds. For references see \cite{CDW,CrFer04}.

We denote the source/target maps of a Lie groupoid $\G\rr M$ by $\mathbf{s},\mathbf{t}:\G\to M$, and the multiplication by $\mathbf{m}:\G\times_{\mathbf{s},\mathbf{t}}\G\to \G$.

A differential form $\eta \in \Omega^q(\G)$ is called {\bf multiplicative} if
\[
 \mathbf{m}^{\ast}\eta=\pr_1^{\ast}\eta+\pr_2^{\ast}\eta\in \Omega^q(\G\times_{\mathbf{s},\mathbf{t}}\G),
\]
where $\pr_1, \pr_2:\G\times_{\mathbf{s},\mathbf{t}}\G\to \G$ are the projections.

A {\bf symplectic groupoid} is a Lie groupoid $\G \rr M$ endowed with a multiplicative symplectic structure $\omega \in \Omega^2(\G)$. The base $M$ of a symplectic groupoid $(\G,\omega)$ carries a Poisson structure $\pi$ so that:
\[(M,\pi)\stackrel{\mathbf{s}}{\lmap}(\G,\omega)\stackrel{\mathbf{t}}{\rmap}(M,-\pi)\]is a full dual pair.

A Poisson manifold $(M,\pi)$ is called {\bf integrable} if such a symplectic groupoid $(\G,\omega)$ exists giving rise to $\pi$, in which case the groupoid is said to {\bf integrate} $(M,\pi)$.

\begin{theorem}\label{thm : integrability_up}
A Poisson transversal $(X,\pi_X)$ of a Poisson manifold $(M,\pi)$ is integrable if, and only if, the restriction $(U,\pi|_U)$ of $\pi$ to an open neighborhood $U$ of $X$ is an integrable Poisson manifold.
\end{theorem}
\begin{proof}
 \underline{Step 1 : If.} Let $(\Sigma,\Omega)\rr (U,\pi)$ be a symplectic groupoid, and $p:U \supset E\to X$ be a tubular neighborhood on which the normal form holds: $\pi|_E=\pi_X^{\sigma}$, for some closed two-form $\sigma$ on $E$, satisfying $\sigma(v)=0$ for all $v\in TX$. Denote $\G_X:=\Sigma|_X$, $\omega_X:=\Omega|_{\G_X}$. Then $\pi_X$ is integrable by the symplectic groupoid $(\G_X,\omega_X)\rr (X,\pi_X)$. This is proved in \cite{CrFer04}, in the more general setting of ``Lie-Dirac submanifolds'' (Theorem 9); for completeness, we include a simple proof:

Applying Lemma \ref{lem : Poisson pulls back Poisson transversals} to the Poisson map
\[(\mathbf{t},\mathbf{s}): (\Sigma,\Omega) \to (U,-\pi)\times (U,\pi),\]
and the Poisson transversal $X \times X \subset U \times U$, we deduce that $(\mathbf{t},\mathbf{s})$ is transverse to $X \times X$, that $(\mathbf{t},\mathbf{s})^{-1}(X \times X)=:\G_X \subset \Sigma$ is a Poisson transversal in $\Sigma$ (thus $\omega_X$ is symplectic), and that the induced map
\[(\mathbf{t},\mathbf{s}): (\G_X,\omega_X) \to (X,-\pi_X)\times (X,\pi_X)\]
is again Poisson. Hence $(\G_X,\omega_X)$ is a symplectic groupoid integrating $(X,\pi_X)$.

\underline{Step 2 : Only if.} Recall \cite{MackenzieXu00} that integrability of a Poisson manifold by a symplectic groupoid is equivalent to integrability of its cotangent Lie algebroid. In particular, $\G_X$ integrates $T^*X$. By Theorem B and Lemma \ref{pro : cotangent=pullback} below, in a tubular neighborhood $p:E\to X$ of the Poisson transversal $(X,\pi_X)\subset (M,\pi)$, the cotangent Lie algebroid $T^*E$ of $\pi|_E$ is isomorphic to the pullback Lie algebroid $TE\times_{TX}T^*X$ of the cotangent Lie algebroid $T^*X$ of $\pi_X$ by $p$. By Proposition 1.3 \cite{Mackenzie87}, the pullback Lie algebroid $TE\times_{TX}T^*X$ is integrable by the pullback groupoid (see below), and so $(E,\pi_X^{\sigma})$ is integrable.
\end{proof}

An inconvenient feature of both Theorem \ref{thm : integrability_up} and its proof is that we are left with a poor understanding of how the symplectic groupoids integrating $(X,\pi_X)$ and a neighborhood of it are related. This is the issue we address in the next section.

\section{Normal form for symplectic groupoids}

Our next goal is to state and prove Theorem \ref{thm : integrability_down} below, which refines Theorem \ref{thm : integrability_up} in that it gives a precise description of the symplectic groupoid integrating a neighborhood of a Poisson transversal in terms of the symplectic groupoid integrating the Poisson transversal itself.

We begin with a description of the Lie algebroid structure corresponding to Poisson structures constructed using the 'Poisson transversal recipe'. Concretely, consider the following set-up, which appears around Poisson transversals:
\begin{itemize}
\item $(X,\pi_X)$ is a Poisson manifold;
\item $p:E\to X$ is a surjective submersion;
\item $\sigma$ is a closed two-form on $E$ so that the Dirac structure $p^*(L_{\pi_X})^{\sigma}$ corresponds to a globally defined Poisson structure $\pi_X^{\sigma}$ on $E$.
\end{itemize}

Consider the pullback of the Lie algebroid $T^*X$ via the submersion $p:E\to X$ (see e.g.\ \cite{MackenzieHiggins90} for the general construction of Lie algebroid pullbacks)
\[TE\times_{TX}T^*X=\{(U,\eta)\in TE\times T^*X : p_*(U)=\pi_X^{\sharp}(\eta)\}.\]

The Lie algebroid $TE\times_{TX}T^*X$ fits into a short exact sequence of Lie algebroids:
\begin{equation}\label{diag : short exact}
0\rmap \mathbb{V}\rmap TE\times_{TX}T^*X\rmap T^*X\rmap 0,
\end{equation}
where $\mathbb{V}\subset TE$ denotes the Lie algebroid $\mathbb{V}=\ker(p_{\ast})$. We have:

\begin{lemma}\label{pro : cotangent=pullback}
The cotangent Lie algebroid $T^*E$ of $\pi_X^{\sigma}$ is isomorphic to the pullback Lie algebroid $TE\times_{TX}T^*X$ via the map
\[\sigma^{\sharp}+p^*:TE\times_{TX}T^*X\diffto T^*E, \ \ (U,\eta)\mapsto \sigma^{\sharp}(U)+p^*(\eta).\]
Under this isomorphism, the short exact sequence (\ref{diag : short exact}) corresponds to
\[0\rmap \mathbb{V}\stackrel{\sigma^{\sharp}}{\rmap} T^*E \rmap T^*X\rmap 0,\]
where the second map assigns to $\xi \in T^*E$ the unique $\eta \in T^*X$ for which $p^*(\eta)=\xi-\sigma^{\sharp}((\pi_X^{\sigma})^{\sharp}(\xi))$.
\end{lemma}

\begin{proof}
We have a sequence of Lie algebroid isomorphisms: first, the cotangent Lie algebroid $T^*E$ of $\pi_X^{\sigma}$ is defined so that the map \[\pr_{T^*E}:p^*(L_{\pi_X})^{\sigma}\diffto T^*E\]
be a Lie algebroid isomorphism; next, the gauge transformation by $\sigma$ is also a Lie algebroid isomorphism
\[e_{\sigma}: p^*(L_{\pi_X})\diffto p^*(L_{\pi_X})^{\sigma}, \ \ e_{\sigma}(U+\xi)=U+\xi+\sigma^{\sharp}(U);\]
and finally, the map $TE\times_{TX}T^*X\diffto p^*(L_{\pi_X})$, $U+\eta \mapsto U+p^*(\eta)$ is an isomorphism as well. The composition of these maps returns the morphisms from the statement.
\end{proof}

We present next a general construction for symplectic groupoids, which provides the local model of a symplectic groupoid around its restriction to a Poisson transversal.

\subsection*{A pullback construction for symplectic groupoids}

The construction of the pullback groupoid is rather standard (according to \cite{Mackenzie87}, it dates back to Ehresmann). We reexamine the construction in the setting of symplectic groupoids, in order to obtain a more explicit proof of Theorem \ref{thm : integrability_up}.

Let $\mathcal{P}(E):=E\times E\rr E$ and $\mathcal{P}(X):=X\times X\rr X$ stand respectively for the pair groupoids of $E$ and $X$. Define the groupoid $\G_{X}^E\rr E$ to be the pullback of the groupoid maps:
\begin{equation}\label{diag : pullback diagram}
 \xymatrix{
  \G_{X}^E \ar[d]_{\mathbf{p}} \ar[rr]^{(\mathbf{t},\mathbf{s})} && \mathcal{P}(E) \ar[d]^{p \times p \ .}\\
  \G_X \ar[rr]_{(\mathbf{t},\mathbf{s})} && \mathcal{P}(X)
 }
\end{equation}
That is, $\G_X^E$ is the manifold
\begin{gather*}
\G_{X}^E:=\left\lb (e',g,e): p(e')=\mathbf{t}(g), p(e)=\mathbf{s}(g) \right\rb \subset E\times \G_X \times E,
 \end{gather*}
endowed with the structure maps
 \begin{gather*}
 \mathbf{s}(e',g,e)=e, \quad \mathbf{t}(e',g,e)=e', \quad (e'',h,e')(e',g,e)=(e'',hg,e)\\
 (e',g,e)^{-1}=(e,g^{-1},e'), \quad \mathbf{1}_e=(e,1_{p(e)},e);
\end{gather*}
As pullbacks by groupoid maps of closed, multiplicative forms $\omega_X \in \Omega^2(\G_X)$, $\sigma \in \Omega^2(\mathcal{P}(E))$, both $\mathbf{p}^*(\omega_X)$ and $\mathbf{s}^*(\sigma)-\mathbf{t}^*(\sigma)$ are closed, multiplicative two-forms on $\G_X^E$, and hence so is their sum:
\[
 \omega_{E} \in \Omega^2(\G_X^E), \qquad \omega_E:=\mathbf{p}^{\ast}(\omega_{X})+\mathbf{s}^{\ast}(\sigma)-\mathbf{t}^{\ast}(\sigma).
\]

\begin{proposition}\label{pro : model presymplectic groupoid}
$(\G_X^E,\omega_{E})\rr (E,\pi_X^{\sigma})$ is a symplectic groupoid.
\end{proposition}

The proof of Proposition \ref{pro : model presymplectic groupoid} uses some general remarks about Dirac structures and Dirac maps:

\begin{lemma}\label{lem : Dirac}Consider a commutative diagram of manifolds:
\[
 \xymatrix{
  A \ar[d]_{i} \ar[r]^{j} & B \ar[d]^k\\
  C \ar[r]_{l} & D,
 }
\]
where $A$ and $TA$ are identified with the set-theoretical pullbacks $A\cong B\times_DC$, and $TA\cong TB\times_{TD} TC$ (e.g., if $k:B\to D$ and $l:C\to D$ are transverse maps). Assume further that the manifolds above are endowed with Dirac structures: $L_A$ on $A$, $L_B$ on $B$, $L_C$ on $C$, and $L_D$ on $D$.
\begin{enumerate}[(a)]
\item  If $k$ and $i$ are backward Dirac maps, and $l$ is forward Dirac, then $j$ is also forward Dirac.
\item If $j:(A,L_A)\to (B,L_B)$ is forward Dirac, and $\omega$ is a closed two-form on $B$, then $j$ is also a forward Dirac map between the gauge-transformed Dirac structures: $j:(A,L_A^{j^*(\omega)})\to (B,L_B^{\omega})$.
\item If $L_A$ is the graph of a closed two-form $\omega$ on $A$, and $L_B$ is the graph of a Poisson structure $\pi$ on $B$, and $j:(A,L_{A})\to (B,L_{B})$ is forward Dirac, then $\ker(\omega)\subset \ker(j_{\ast})$.
\end{enumerate}
\end{lemma}
\begin{proof}
(a) Observe that, counting dimensions, it suffices to show that $j_{\ast}L_{A} \subset L_B$. Fix then $a \in A$, and set $b:=j(a)$, $c:=i(a)$ and $d:=k(b)=l(c)$. To further simplify the notation, we also let $L_a:=L_{A,a}$, $L_b:=L_{B,b}$, $L_c:=L_{C,c}$, and $L_d:=L_{D,d}$.

Choose $X_B+\eta_B \in j_*(L_{a})$. This means that $X_B=j_*(X_A)$, for some vector $X_A$ with $X_A + j^*(\eta_B)\in L_{a}$. Since $i$ is a backward Dirac map, there is a covector $\eta_C$ so that $j^*(\eta_B)=i^*(\eta_C)$ and $i_*(X_A)+\eta_C\in L_c$. Since $i^*(\eta_C)=j^*(\eta_B)$, the dual of the pullback property for $TA$ implies that there is a covector $\eta_D\in T^*_dD$, with $\eta_C=l^*(\eta_D)$ and $\eta_B=k^*(\eta_D)$. Since $l$ is a forward Dirac map, we have that $l_*(X_C)+\eta_D \in L_d$. Commutativity of the diagram implies that $l_*(X_C)=k_*(X_B)$. Thus $k_*(X_B)+\eta_D \in L_d$, and $k^*(\eta_D)=\eta_B$. Finally, since $k$ is a backward Dirac map, $X_B + \eta_B \in L_b$. Hence $X_B+\eta_B \in L_b$, and the conclusion follows.

\noindent (b) Note that, again by dimensional reasons, we need only show that $L_b^{\omega}\subset j_*(L_a^{j^*(\omega)})$. Choose $a\in A$ and set $b:=j(a)$, $L_a:=L_{A,a}$ and $L_B:=L_{B,b}$. Consider $X_B + \eta_B \in L_b^{\omega}$. This means that $X_B + \eta_B-\iota_{X_B}\omega \in L_b$. Since $j$ is a forward Dirac map, there is a vector $X_A$ with $X_B=j_*(X_A)$ and $X_A + j^*(\eta_B-\iota_{X_B}\omega)\in L_a$. Clearly, $j^*(\iota_{X_B}\omega)=j^*(\iota_{j_*(X_A)}\omega)=\iota_{X_A}j^*(\omega)$. Hence $X_A + j^*(\eta_B)-\iota_{X_A}j^*(\omega) \in L_a$, and so $X_A+ j^*(\eta_B) \in L_a^{j^*(\omega)}$. This shows that $X_B + \eta_B \in j_*(L_a^{j^*(\omega)})$.

\noindent (c) If $V\in \ker(\omega)$, then $V \in L_{\omega}$. But $j$ forward Dirac implies $j_*(V) \in L_{\pi}$, and therefore $j_*(V)=0$.
 \end{proof}

\begin{proof}[Proof of Proposition \ref{pro : model presymplectic groupoid}]
We apply Lemma \ref{lem : Dirac} (a) to the pullback diagram (\ref{diag : pullback diagram}), where these manifolds have the following Dirac structures : \[(X,-\pi_X)\times (X,\pi_X),\ (E, p^*(L_{-\pi_X}))\times (E, p^*(L_{\pi_X})),\  (\G_X,\omega_X),\ (\G_X^E, \mathbf{p}^*(\omega_X)).\]
We deduce that the map
\[(\mathbf{t},\mathbf{s}):(\G_X^E,\mathbf{p}^*(\omega_X))\rmap (E, p^*(L_{-\pi_X}))\times (E, p^*(L_{\pi_X}))\]
is forward Dirac. By Lemma \ref{lem : Dirac} (b), this map is forward Dirac also after gauge-transformations: \[(\mathbf{t},\mathbf{s}):(\G_X^E, \omega_E)\rmap (E,-\pi_X^{\sigma})\times (E,\pi_X^{\sigma}).\]
It remains to show that $\omega_E$ is nondegenerate. As $\G_X^E=E\times_X\G_X\times_X E$, we obtain that its tangent bundle is the pullback $T\G_X^E=TE\times_{TX}T\G_X\times_{TX} TE$. Explicitly:
\[T\G_X^E=\{(A,B,C)\in TE\times T\G_X\times TE : p_{\ast}(A)=\mathbf{t}_{\ast}(B), \mathbf{s}_{\ast}(B)=p_{\ast}(C)\}.\]
In this decomposition, we can write
\begin{equation}\label{eq : formula omega_E}
\omega_E((A,B,C),(A',B',C'))=-\sigma(A,A')+\omega_X(B,B')+\sigma(C,C').
\end{equation}
By Lemma \ref{lem : Dirac} (c),
\[\ker(\omega_E)\subset \ker(\mathbf{s}_{\ast})\cap\ker(\mathbf{t}_{\ast})=\{(0,B,0) : \mathbf{s}_{\ast}(B)=0, \ \mathbf{t}_{\ast}(B)=0\}\]
But, for $(0,B,0)\neq 0$ we have that $\iota_{(0,B,0)}\omega_E=\mathbf{p}^*(\iota_B\omega_X)\neq 0$, because $\omega_X$ is nondegenerate. Hence $\omega_E$ is nondegenerate. Thus $(\G_X^E,\omega_E)$ is a symplectic groupoid integrating $(E,\pi_X^{\sigma})$.
\end{proof}

\subsection*{The normal form theorem}

We are now ready to prove that the structure of a symplectic groupoid around a Poisson transversal is described by the pullback construction:

\begin{theorem}[Normal form for symplectic groupoids]\label{thm : integrability_down}
Let $(\Sigma,\Omega)\rr (M,\pi)$ be a symplectic groupoid, and let $(X,\pi_X)$ be a Poisson transversal in $M$. Let $p:E\to X$ be a tubular neighborhood on which the normal form holds: $\pi|_E=\pi_X^{\sigma}$, for some closed two-form $\sigma$ on $E$, satisfying $\sigma^{\sharp}(U)=0$ for all $U\in TX$. Denote
\[\G_X:=\Sigma|_X,\ \ \omega_X:=\Omega|_{\G_X}, \ \  \Sigma_E:=\Sigma|_E, \ \ \Omega_E:=\Omega|_{\Sigma_E}.\]
Then the Lie algebroid isomorphism $TE\times_{TX}T^*X\cong T^*E$ described in Lemma \ref{pro : cotangent=pullback} integrates to an isomorphism of symplectic groupoids $\Psi:(\G_X^E,\omega_E)\cong (\Sigma_E,\Omega_E)$.
\end{theorem}

\begin{proof}We split the proof into three steps: constructing $\Psi$ as an isomorphism of Lie groupoids, showing that it is a symplectomorphism, and finally, that it integrates the isomorphism of Lie algebroids $TE\times_{TX}T^*X\cong T^*E$.

\smallskip

\noindent\emph{Step 1: Construction of Lie groupoid isomorphism $\Psi$.}

Let $A$ denote the Lie algebroid of $\Sigma_E$, i.e.\
\[T\Sigma_E|_E=TE\oplus A, \ \ A=\ker(\mathbf{s}_{\ast}).\]
The identification between the Lie algebroid $A$ and the cotangent Lie algebroid $T^*E$ is obtained via the symplectic form:
\[-\Omega_E^{\sharp}:A\diffto T^*E, \ \ -\Omega_E^{\sharp}(u)(v)=-\Omega_E(u,v).\]

By Lemma \ref{pro : cotangent=pullback}, the map $\sigma^{\sharp}:\mathbb{V}\to T^*E$ is an injective Lie algebroid morphism. Note that $\mathbb{V}$ is integrable by the submersion groupoid $E\times_XE\rr E$ of $p:E\to X$. Since $E\times_XE$ has 1-connected $\mathbf{s}$-fibres, the Lie algebroid map
\begin{equation}\label{eq : Lie algebroid map}
(-\Omega_E^{\sharp})^{-1}\circ \sigma^{\sharp}:\mathbb{V}\to A
\end{equation}
integrates to a Lie groupoid map
\[\Phi:E\times_X E\rmap \Sigma_E.\]
For $e\in E$, denote by $\tau(e)\in E\times_X E$ the arrow that starts at $p(e)\in X\subset E$, and ends at $e$: $\tau(e):=(e,p(e))$, and define the map:
\[\Psi: \G_X^E\rmap \Sigma_E, \ \ \Psi(e',g,e):=\Phi(\tau(e'))\cdot g\cdot \Phi(\tau(e))^{-1}.\]
It is straightforward to check that $\Psi$ is an isomorphism of Lie groupoids, with inverse
\[\Theta: \Sigma_E\rmap \G_X^E, \ \ \Theta(\overline{g})=(e',\Phi(\tau(e'))^{-1}\cdot\overline{g}\cdot\Phi(\tau(e)),e),\]
\[\textrm{where} \ \ \ e':=\mathbf{t}(\overline{g}), \ \ e:=\mathbf{s}(\overline{g}).\]

\smallskip

\noindent\emph{Step 2: $\Psi$ is an isomorphism of symplectic groupoids.} We begin with the observation that the identification $TE\times_{TX} T\G_X\times_{TX}TE= T\G_X^E$ can be realized using the multiplication map:
\[TE\times_{TX} T\G_X\times_{TX}TE \ni (U,V,W)\mapsto \textbf{m}_{\ast}\left(\textbf{m}_{\ast}(\tau_{\ast}(U),V),\tau^{-1}_{\ast}(W)\right)\in T\G_X^E.\]
Therefore, for any multiplicative two-form $\eta$, we have that:
\begin{align*}
&\eta((U,V,W),(U',V',W'))=\\
&=\eta|_{E\times_XE}(\tau_{\ast}(U),\tau_{\ast}(U'))+\eta|_{\G_X}(V,V')+\eta|_{E\times_XE}(\tau^{-1}_{\ast}(W),\tau^{-1}_{\ast}(W')).
\end{align*}
Therefore, in order to prove that $\omega_E$ and $\widetilde{\omega}_E:=\Psi^*(\Omega_E)$ coincide, it suffices to show that they have the same restriction to the subgroupoids $\mathcal{G}_X,E \times_X E$, which is what we turn to next.

That $\omega_E|_{\mathcal{G}_X}=\widetilde{\omega}_E|_{\mathcal{G}_X}$ follows by our construction: indeed, we have $\Psi|_{\mathcal{G}_X}=\textrm{id}$ and $\omega_X=\Omega_E|_{\mathcal{G}_X}$; since $\sigma|_X=0$, also $\omega_X=\omega_E|_{\mathcal{G}_X}$, and hence our conclusion.

We next show that $\omega_E|_{E\times_XE}=\widetilde{\omega}_E|_{E\times_XE}$. Regarding $E\times_XE$ as the subgroupoid of $\mathcal{G}_X^E$ consisting of elements $(e', 1_x, e)$, for $p(e')=x=p(e)$, we clearly have $\Psi|_{E\times_XE}=\Phi$, and
\[\omega_E|_{E\times_XE}=\mathbf{s}^*(\sigma)-\mathbf{t}^*(\sigma).\]

Now, $\Phi^*(\Omega_E)$ is a multiplicative two-form on the source-simply connected groupoid $E\times_XE$, and is thus determined by its IM-form \cite{BursCabr12}. The IM-form\footnote{Note that our sign convention is different from that in \cite{BursCabr12}; namely, the IM form corresponding to a closed two-form $\eta$ on a groupoid $\G$, is given by $A\ni V\mapsto \mathbf{u}^*(-\iota_V\eta)$, where $\mathbf{u}:M\to \G$ is the unit map.} corresponding to $\Omega_E$ is simply $-\Omega_E^{\sharp}:A\to T^*E$. Pulling it back via the Lie algebroid map (\ref{eq : Lie algebroid map}) to $\mathbb{V}$, we deduce that the IM-form corresponding to $\Phi^*(\Omega_E)$ is $\sigma^{\sharp}:\mathbb{V}\to T^*E$, which is also the IM form of the multiplicative two-form $\textbf{s}^*(\sigma)-\textbf{t}^*(\sigma)$. We thus conclude that $\textbf{s}^*(\sigma)-\textbf{t}^*(\sigma)=\Phi^*(\Omega_E)$, that
\[\omega_E|_{E\times_XE}=\widetilde{\omega}_E|_{E\times_XE},\]
and that $\Psi$ is an isomorphism of symplectic groupoids.
\smallskip

\noindent\emph{Step 3: $\Psi$ integrates the Lie algebroid isomorphism of Lemma \ref{pro : cotangent=pullback}.} Note that the algebroid of $\mathcal{G}_X$ is given by
\[T\G_X|_X=TX\oplus A_X, \ \ A_X:=\{v\in A: \textbf{t}_{\ast}(v)\in TX\}\subset A|_X; \ \ \]
and the identification of $A_X$ with the cotangent Lie algebroid of $\pi_X$ is given by $(-\omega_X)^{\sharp}:A_X\diffto T^*X$.
Now, the Lie algebroid of $\G_X^E$ is the pullback Lie algebroid $TE\times_{TX}A_X$. Let $\psi$ denote the Lie algebroid map induced by $\Psi$. Consider the commutative diagram of Lie algebroid isomorphisms:
\[
 \xymatrix{
  TE\times_{TX}A_X \ar[d]_{(\textrm{id},(-\omega_X)^{\sharp})}\ar[dr]^{(-\omega_E)^{\sharp}} \ar[r]^{\phantom{12345}\psi} & A \ar[d]^{(-\Omega_{E})^{\sharp}}\\
  TE\times_{TX} T^*X \ar[r]_{\phantom{12345}\varphi } & T^*E,
 }
\]
where the top-right triangle is commutative by the fact that $\Psi^*(\Omega_E)=\omega_E$, and $\varphi$ is defined such that the entire diagram is commutative, i.e.\ \[\varphi:=(\textrm{id},(-\omega_X)^{\sharp})\circ(-\omega_E^{\sharp})^{-1}.\]
We need to check that $\varphi$ is the isomorphism from Lemma \ref{pro : cotangent=pullback}, and for that we need to compute $(\omega_E)^{\sharp}$. At a unit $e\in E$, with $x=p(e)$, there are two decompositions of the tangent space to $\G_X^E$:
\[T_{e}\G_X^E\cong T_eE\oplus T_eE\times_{T_xX}A_{X,x}\cong T_eE\times_{T_xX}T_x\G_X\times_{T_xX} T_eE.\]
In the first decomposition, the first factor is the tangent space to the units, and the second is the Lie algebroid (i.e.\ the tangent space to the source-fibre), whereas the second decomposition is based on the pullback construction of $\G_X^E=E\times_X\G_X\times_X E$. The identification between these decompositions is given by:
\[(U,V,W)\mapsto (U+V,p_*(U)+W,U).\]
Using the expression (\ref{eq : formula omega_E}) of $\omega_E$ with respect to the second decomposition, and the identification above, we compute $\omega_E$ with respect to the first decomposition:
\begin{align*}
\omega_E((0,V,W),(U,0,0))&=\left(s^*(\sigma)-t^*(\sigma)+\mathbf{p}^*(\omega_X)\right)(V,W,0)(U,p_*(U),U)=\\
&=-\sigma(V,U)+\omega_X(W,p_*(U));
\end{align*}
therefore:
\[(-\omega_E)^{\sharp}(V,W)=\sigma(V)-p^*((\omega_X)^{\sharp}(W)).\]
This shows that the diagram commutes for $\varphi(V,\eta)=\sigma^{\sharp}(V)+p^*(\eta)$, which is the map from Lemma \ref{pro : cotangent=pullback}. This finishes the proof.
\end{proof}

\section{Linear Poisson structures}

In this section we write our results explicitly for linear Poisson structures. Our goal is to illustrate Theorems A, B, 1 and 3 in this context, thus recasting and reproving some well-known results in what (we would argue) is their proper setting.

\medskip

Let $(\mathfrak{g},[\cdot,\cdot])$ be a Lie algebra. The dual vector space $\mathfrak{g}^*$ carries a canonical
Poisson structure $\pi_{\mathfrak{g}}$, called the linear Poisson structure. It is defined by
\[\pi_{\mathfrak{g},\xi}:=\xi\circ [\cdot,\cdot]\in \wedge^2\mathfrak{g}^*=\wedge^2T_{\xi}\mathfrak{g}^*.\]
In fact, any Poisson structure on a vector space for which the linear functions form a Lie subalgebra is of this form.

Linear Poisson structures are always integrable. The following construction of a symplectic groupoid integrating $(\mathfrak{g}^*,\pi_{\mathfrak{g}})$ is standard and we recall it to establish the notation. Let $G$ be a Lie group integrating $\mathfrak{g}$. Then a symplectic groupoid integrating $\pi_{\mathfrak{g}}$ is the action groupoid:
\[\left(G\ltimes \mathfrak{g}^*,\Omega_{G}\right)\rightrightarrows(\mathfrak{g}^*,\pi_{\mathfrak{g}})\]
associated to the coadjoint action $(g,\xi)\mapsto\mathrm{Ad}_{g^{-1}}^*\xi$; it carries the symplectic structure:
$\Omega_G\in \Omega^2(G\times\mathfrak{g}^*)$ given by:
\begin{equation}\label{eq : explicit form of Omega_G}
\Omega_G((x,\xi),(y,\eta))_{(g,\xi_0)}=\xi(g^{-1}y)-\eta(g^{-1}x)+\xi_0([g^{-1}x,g^{-1}y]),
\end{equation}
for $(x,\xi),(y,\eta)\in T_{(g,\xi_0)}(G\times\mathfrak{g}^*)=T_gG\times \mathfrak{g}^*$, where $g^{-1}x$ and $g^{-1}y$ denotes the action of $G$ on $TG$. For a detailed exposition (with similar notation) see e.g. \cite[Section 2.4.2]{Thesis}.

\begin{illustration1}
 \begin{enumerate}[a)]
  \item The Poisson manifold $(\mathfrak{g}^*,\pi_{\mathfrak{g}})$ carries a canonical, complete Poisson spray $\mathcal{V}_{\mathfrak{g}}$, whose flow (under the identification $T^*\mathfrak{g}^*=\mathfrak{g}\times\mathfrak{g}^*$) is given by:
  \[\phi_t:T^*\mathfrak{g}^*\rmap T^*\mathfrak{g}^*,\ \ (x,\xi)\mapsto(x,e^{-t\mathrm{ad}_x^*}\xi).\]
  \item Let $\mathcal{O}(\mathfrak{g}) \subset \mathfrak{g}$ be the subspace where the Lie-theoretic exponential map $\exp:\mathfrak{g} \to G$ is a local diffeomorphism. Then the closed two-form:
  \[\Omega_{\mathfrak{g}}:=\int_0^1\phi_t^*\omega_{\mathrm{can}}\mathrm{d}t \\ \in \\ \Omega^2(T^{\ast}\mathfrak{g}^{\ast})\]is symplectic exactly on $\mathcal{O}(\mathfrak{g}) \times \mathfrak{g}^* \subset T^*\mathfrak{g}^*$, and gives rise to the full dual pair:
  \[
  \xymatrix{
  (\mathfrak{g}^*,\pi_{\mathfrak{g}}) &(\mathcal{O}(\mathfrak{g})\times \mathfrak{g}^*,\Omega_{\mathfrak{g}})  \ar[l]_<<<<{\mathrm{pr}_1} \ar[r]^>>>>>{\exp_{\mathcal{V}_{\mathfrak{g}}}} & (\mathfrak{g}^*,-\pi_{\mathfrak{g}})
  }
  \]
  Explicitly:
  \begin{equation}\label{eq : explicit 2 form infi}
\Omega_{\mathfrak{g}}((x,\xi),(y,\eta))_{(x_0,\xi_0)}=\xi(\Xi_{x_0}y)-\eta(\Xi_{x_0}x)+\xi_0\left(\left[\Xi_{x_0}x,\Xi_{x_0}y\right]\right),
\end{equation}
where $\Xi_{x_0}$ is the linear endomorphism of $\mathfrak{g}$ given by:
\begin{equation*}
\Xi_{x_0}(x)=\int_{0}^1e^{-t\mathrm{ad}_{x_0}}(x)\mathrm{d}t=\frac{e^{-\mathrm{ad}_{x_0}}-\mathrm{Id}_{\mathfrak{g}}}{-\mathrm{ad}_{x_0}}(x).
\end{equation*}
\item $X\subset \mathfrak{g}^*$ is a Poisson transversal if and only if, for every $\lambda\in X$, the two-form $\lambda \circ [\cdot,\cdot]$ is nondegenerate on the annihilator of $T_{\lambda}X$. Moreover, under the identification $N^*X=\bigcup_{\lambda\in X} N_{\lambda}^{\ast}X\times\{\lambda\}\subset \mathfrak{g}\times X$, we have a Poisson diffeomorphism in a neighborhood of $X$ given by:
\[\exp_{\mathcal{V}_{\mathfrak{g}}}:(N^*X,\pi_X^{-\Omega_{\mathfrak{g}}|_{N^*X}})\to (\mathfrak{g}^*,\pi_{\mathfrak{g}}) \ \ (x,\lambda)\mapsto e^{-\mathrm{ad}_x^*}\lambda;\]
\item If $f:\mathfrak{g} \to \mathfrak{h}$ is a Lie algebra map, and $Y \subset \mathfrak{h}^*$ is a Poisson transversal, then $X:=(f^*)^{-1}Y \subset \mathfrak{g}^*$ is a Poisson transversal, and $f$ induces a bundle map $F$ fitting into the commutative diagram of Poisson maps:
\[\xymatrix{
 (\mathfrak{h}^*,\pi_{\mathfrak{h}}) \ar[r]^{f^*} & (\mathfrak{g}^*,\pi_{\mathfrak{g}}) \\
 (N^*Y,\pi_{Y}^{-\Omega_{\mathfrak{h}}|_{N^*Y}}) \ar[u]^{\exp_{\V_{\mathfrak{g}}}} \ar[r]_F & (N^*X,\pi_X^{-\Omega_{\mathfrak{g}}|_{N^*X}})  \ar[u]_{\exp_{\V_{\mathfrak{h}}}}}\]
 \end{enumerate}
\end{illustration1}
\begin{proof}[Proof of Illustration 1]
\begin{enumerate}[a)]
 \item The flow $\phi_t$ in the statement has infinitesimal generator the vector field $\mathcal{V}_{\mathfrak{g}}\in \mathfrak{X}^1(\mathfrak{g}\times \mathfrak{g}^*)$ given by:
\[\mathcal{V}_{\mathfrak{g},(x,\xi)}:=(0,-\mathrm{ad}_x^*\xi)=(0,\pi_{\mathfrak{g},\xi}^{\sharp}x)\in \mathfrak{g}\times \mathfrak{g}^*=T_{x}\mathfrak{g}\times T_{\xi}\mathfrak{g}^*,\]
where $\mathrm{ad}_xy=-[x,y]$ since we use right invariant vector fields to define the Lie bracket, and this is clearly a spray.
 \item Since trajectories $\phi_t(x,\xi)$ of $\mathcal{V}_{\mathfrak{g}}$ are cotangent paths, they can be integrated to elements in the Lie groupoid, yielding a \emph{groupoid exponential map}:
\[\mathrm{Exp}_{\mathcal{V}_{\mathfrak{g}}}:T^*\mathfrak{g}^*\rmap G\ltimes \mathfrak{g}^*, \ \ (x,\xi)\mapsto(\exp(x),\xi),\]where $\exp:\mathfrak{g} \to G$ denotes the \emph{Lie-theoretic exponential map}.

On the other hand, the \emph{spray exponential map} $\exp_{\mathcal{V}_{\mathfrak{g}}}$, i.e., the composition of $\phi_1$ with the bundle projection $T^*\mathfrak{g}^* \to \mathfrak{g}^*$, becomes $\mathrm{Exp}_{\mathcal{V}_{\mathfrak{g}}}$ composed with the target map:
\[\exp_{\mathcal{V}_{\mathfrak{g}}}(x,\xi)=e^{-\mathrm{ad}_x^*}\xi.\]

Now, the pullback by $\mathrm{Exp}_{\mathcal{V}_{\mathfrak{g}}}$ of the symplectic structure $\Omega_G$ of (\ref{eq : explicit form of Omega_G}) is given by the formula in Theorem A (see \cite{CrFer04} for details); hence the general considerations above imply that the two-form $\Omega_{\mathfrak{g}}$ is given by:
\begin{equation}\label{eq : equality of forms}
\Omega_{\mathfrak{g}}=(\mathrm{Exp}_{\mathcal{V}_{\mathfrak{g}}})^*\Omega_{G}.
\end{equation}This implies that $\Omega_{\mathfrak{g}}$ is nondegenerate exactly on $\mathcal{O}(\mathfrak{g})\times\mathfrak{g}^*$, and that the following is a commutative diagram of Poisson maps:
\begin{equation}\label{diag : diagram with exps}
\xymatrix{
&(G\times \mathfrak{g}^*,\Omega_{G}) \ar[ld]_s \ar[rd]^{t}\\
(\mathfrak{g}^*,\pi_{\mathfrak{g}}) &(\mathcal{O}(\mathfrak{g})\times \mathfrak{g}^*,\Omega_{\mathfrak{g}})  \ar[l]^<<<<{\mathrm{pr}_1} \ar[r]_>>>>>{\exp_{\mathcal{V}_{\mathfrak{g}}}}\ar[u]|{\exp\times\mathrm{Id}_{\mathfrak{g}^*}} & (\mathfrak{g}^*,-\pi_{\mathfrak{g}})
}
\end{equation}
The explicit formula (\ref{eq : explicit 2 form infi}) for $\Omega_{\mathfrak{g}}$ is obtained by pulling back $\Omega_G$ from (\ref{eq : explicit form of Omega_G}), and we conclude with the observation that the linear endomorphism $\Xi_{x_0}:\mathfrak{g} \to \mathfrak{g}$ is the left translation of the differential of $\exp:\mathfrak{g}\to G$ at $x_0$, and therefore it is invertible precisely on $\mathcal{O}(\mathfrak{g})$.
\item Consider an affine subspace
$\lambda+L$ passing through a point $\lambda\in \mathfrak{g}^*$ and with direction a linear subspace $L\subset \mathfrak{g}^*$. Now, $\lambda+L$ is a Poisson transversal in a neighborhood of $\lambda$ if and only if the following condition is satisfied:
\begin{equation}\label{eq : PT condition, linear structures 1}
\mathfrak{g}^*=L\oplus L^{\circ}\cdot \lambda, \ \ L^{\circ}\cdot \lambda:=\{X\cdot \lambda : X\in L^{\circ}\};
\end{equation}equivalently:
\begin{equation}\label{eq : PT condition, linear structures 2}
\lambda\circ [\cdot,\cdot]|_{L^{\circ}\times L^{\circ}} \textrm{ is a non-degenerate 2-form on }L^{\circ}.
\end{equation}The remaining claims are immediate.
\item The dual map $f^*:(\mathfrak{h}^*,\pi_{\mathfrak{h}})\rmap (\mathfrak{g}^*,\pi_{\mathfrak{g}})$ to a Lie algebra map $f$ is a Poisson map, hence by Lemma \ref{lem : Poisson pulls back Poisson transversals}, $f^*$ is transverse to $X$, $Y:=(f^*)^{-1}(X)$ is a Poisson transversal in $\mathfrak{h}^*$, and $f^*$ restricts to a Poisson map
\[ f^*|_{Y}:(Y,\pi_{Y})\rmap (X,\pi_{X}).\]
Moreover, $f$ restricts to a linear isomorphism between the conormal spaces $f:N^*_{f^*(\mu)}X \diffto N^*_{\mu}Y$, for all $\mu \in Y$. The inverses \[F_{\mu}=\left(f|_{L_{f^*(\mu)}^{\circ}}\right)^{-1}:L_{\mu}^{\circ}\diffto L_{f^*(\mu)}^{\circ}\]can be put together in a vector bundle map $F:N^*Y \diffto N^*X$ covering $f^*: Y \to X$, which is fibrewise a linear isomorphism.

We conclude by showing that the diagram in the statement commutes. Let $(y,\xi)\in N^*Y$. Then $F(y,\xi)=(x,f^*(\xi))\in N^*X$, where $x$ satisfies $y=f(x)$. For any $z\in\mathfrak{g}$, we have:
 \begin{align*}
 \exp_{\mathcal{V}_{\mathfrak{g}}}\left(F(y,\xi)\right)(z)=\exp_{\mathcal{V}_{\mathfrak{g}}}((x,f^*\xi))(z)=\left(e^{-\mathrm{ad}_x^*}f^*\xi\right)(z)=\\
 =\xi(f(e^{-\mathrm{ad}_x}z))=\xi(e^{-\mathrm{ad}_{f(x)}}f(z))=\xi(e^{-\mathrm{ad}_{y}}f(z))=\\
 =f^*(e^{-\mathrm{ad}_y^*}\xi)(z)=f^*(\exp_{\mathcal{V}_{\mathfrak{g}}}(y,\xi))(z),
 \end{align*}
where we have used that $f$ is a Lie algebra map. Since $f^*$ and the vertical maps are Poisson maps, it follows that also $F$ is Poisson around $Y$.\qedhere
\end{enumerate}
\end{proof}

Our next illustration concerns the specialization of Theorem \ref{thm : integrability_down} for Poisson transversals complementary to coadjoint orbits, in the particularly convenient setting where the coadjoint action is \emph{proper} at the orbit.

\begin{illustration2}
Let $\mathfrak{g}$ be a Lie algebra satisfying the following \emph{splitting condition} at $\lambda\in \mathfrak{g}^*$: there is a decomposition
\begin{equation}\label{eq: spilitting condition (implying linearizability)}
\mathfrak{g}=\mathfrak{g}_{\lambda}\oplus c,
\end{equation}where $\mathfrak{g}_{\lambda}$ the isotropy Lie algebra at $\lambda$, satisfying $[\mathfrak{g}_{\lambda},c]\subset c$. Then:
\begin{enumerate}[a)]
 \item Along $\widetilde{X}:=\lambda+\mathfrak{g}_{\lambda}^*$, the Poisson tensor $\pi_{\mathfrak{g}}$ decomposes as:
 \begin{equation}\label{eq : decomposition of affine Lie-Dirac}
(\lambda+\xi)\circ \pi_{\mathfrak{g}}=\xi\circ \pi_{\mathfrak{g}_{\lambda}}+(\lambda+\xi)\circ \pi_c\in \wedge^2\mathfrak{g}_{\lambda}^*\oplus \wedge^2c^*
\end{equation}where we identify $\mathfrak{g}_{\lambda}^*=c^{\circ}$;
\item $\widetilde{X}$ intersects all coadjoint orbits cleanly and symplectically, and hence inherits an induced Poisson structure $\pi_{\widetilde{X}}$;
\item $\pi_{\widetilde{X}}$ is globally linearizable through the Poisson isomorphism: \[\tau_{\lambda}:(\mathfrak{g}_{\lambda}^*,\pi_{\mathfrak{g}_{\lambda}})\diffto (\widetilde{X}, \pi_{\widetilde{X}}), \ \ \tau_{\lambda}(\xi)=\xi+\lambda;\]
\item The subspace $X \subset \widetilde{X}$ where $\widetilde{X}$ is a Poisson transversal contains $\lambda$, and for a product neighborhood of the origin $V\times W\subset c\times\mathfrak{g}^*_{\lambda}$, the following map is an open Poisson embedding onto a neighborhood of $\lambda$:
\begin{equation}\label{eq : open embedding}
\left(V\times W, \pi_{\mathfrak{g}_{\lambda}}^{\sigma_{\lambda}}\right)\hookrightarrow \left(\mathfrak{g}^*,\pi_{\mathfrak{g}}\right), \ \ (x,\xi)\mapsto e^{-\mathrm{ad}_x^*}(\lambda+\xi),
\end{equation}
where $\sigma_{\lambda}$ is the pullback of $-\Omega_{\mathfrak{g}}$ via the map:
\[c\times\mathfrak{g}^*_{\lambda}\to \mathfrak{g}\times \mathfrak{g}^*, \ \ (x,\xi)\mapsto (x,\lambda+\xi);\]
\item If a Lie group $G$ integrating $\mathfrak{g}$ acts properly at $\lambda$, and $G_{\lambda}$ denotes the isotropy group at $\lambda$, then, by shrinking $W\subset \mathfrak{g}_{\lambda}^*$ if need be, the restriction of the symplectic groupoid $G\ltimes\mathfrak{g}^*$ to the image of the map (\ref{eq : open embedding}) is isomorphic to the product of the groupoid $G_{\lambda}\ltimes W\rr W$ with the pair groupoid $V\times V\rr V$, with symplectic structure:
\[\left(V\times\left(G_{\lambda}\ltimes W\right)\times V, \mathbf{s}^*(\sigma_{\lambda})+\mathbf{p}^*(\Omega_{G_{\lambda}})-\mathbf{t}^*(\sigma_{\lambda})\right) \rr \left(V\times W,\pi_{\mathfrak{g}_{\lambda}}^{\sigma_{\lambda}}\right),\ \ \textrm{where}\]
\[\mathbf{s}(y,(g,\xi),x)=(x,\xi), \ \ \mathbf{p}(y,(g,\xi),x)=(g,\xi), \ \
\mathbf{t}(y,(g,\xi),x)=(y,\mathrm{Ad}_{g^{-1}}^*\xi).\]
\end{enumerate}
\end{illustration2}

\begin{remark}\rm
It was first proved in \cite{Molino} that the splitting condition (\ref{eq: spilitting condition (implying linearizability)}) implies that the transverse Poisson structure to the coadjoint orbit at $\lambda$ is linearizable, see also \cite{Wein85E}.

Submanifolds which intersect the symplectic leaves cleanly and symplectically, and for which the induced bivector is smooth, are called {\bf Poisson-Dirac} \cite{CrFer04}. In fact, the affine submanifold $\lambda+\mathfrak{g}_{\lambda}^*$ turns out to be a {\bf Lie-Dirac} submanifold (also called ``Dirac submanifold''), see \cite[Example 2.18]{Xu}.
\end{remark}
\begin{proof}[Proof of Illustration 2]
 Since $[\mathfrak{g}_{\lambda},\mathfrak{g}_{\lambda}]\subset \mathfrak{g}_{\lambda}$ and $[\mathfrak{g}_{\lambda},c]\subset c$ we have that, in the decomposition:
\[\pi_{\mathfrak{g}}=\pi_{\mathfrak{g}_{\lambda}}+\pi_{m}+\pi_c,\]
corresponding to (\ref{eq : decomposition of affine Lie-Dirac}), the components satisfy:
\[
\pi_{\mathfrak{g}_{\lambda}}\in \mathfrak{g}_{\lambda}\otimes \wedge^2 \mathfrak{g}_{\lambda}^*, \ \ \ \ \pi_{m}\in c\otimes (\mathfrak{g}_{\lambda}^*\otimes c^*), \ \ \ \
\pi_c\in\mathfrak{g}\otimes  \wedge^2 c^*.\]
The fact that $\mathfrak{g}_{\lambda}$ is precisely the isotropy Lie algebra at $\lambda$ is equivalent to:
\begin{equation}\label{eq : conditions for isotropy}
\lambda\circ \pi_{\mathfrak{g}_{\lambda}}=\lambda\circ \pi_m=0, \ \ \lambda\circ \pi_{c}\in \wedge^2c^* \textrm{ is nondegenerate}.
\end{equation}
Hence, on the affine space $\widetilde{X} = \lambda+\mathfrak{g}_{\lambda}^*$ the Poisson bivector takes the form (\ref{eq : decomposition of affine Lie-Dirac}). This proves a), from which b) and c) follow.

The claim in d) that $X \subset \widetilde{X}$ contains $\lambda$ is immediate. Write $X = \lambda+U$, where
\[U:=\left\{\xi\in \mathfrak{g}_{\lambda}^* : (\lambda+\xi)\circ\pi_{c}\in \wedge^2c^* \textrm{ is nondegenerate}\right\}\subset\mathfrak{g}_{\lambda}^*.\]
Observe that $N^{\ast}X=c\times X$ and, by part a), $NX=c^*\times X$. We thus recognize in (\ref{eq : decomposition of affine Lie-Dirac}) the decomposition (\ref{eq : splitting}) of $\pi_{\mathfrak{g}}$ along the Poisson transversal $X$ into tangential and normal components. The remaining claim in d) is the conclusion of Theorem B around $\lambda$, for a product neighborhood $V\times W\subset c\times\mathfrak{g}^*_{\lambda}$ of the origin with $W\subset U$.

As for e), note that the properness assumption implies that the group $G_{\lambda}$ is compact, that the coadjoint orbit through $\lambda$ is closed, and that the splitting (\ref{eq: spilitting condition (implying linearizability)}) can be assumed to be $G_{\lambda}$-invariant. This assumption not only implies that the transverse Poisson structure is linearizable, but also that the Poisson manifold $(\mathfrak{g}^*,\pi_{\mathfrak{g}})$ is linearizable around the coadjoint orbit through $\lambda$ in the sense of \cite{Vorobjev}, see \cite[Example 2.7]{CrMar12}.

By the slice theorem for proper group actions, one can assume (by shrinking $W\subset \mathfrak{g}_{\lambda}^*$ if need be) that $\lambda+W$ is $G_{\lambda}$-invariant, and that its saturation is $G$-equivariantly diffeomorphic to $G\times_{G_{\lambda}}W$ via the map $[g,\xi]\mapsto \mathrm{Ad}_{g^{-1}}^*(\lambda+\xi)$. In particular, this implies that the restriction of the action groupoid $G\ltimes \mathfrak{g}^*$ to $\lambda+W$ is isomorphic to the restriction of the action groupoid $G_{\lambda}\ltimes \mathfrak{g}_{\lambda}^*$ to $W$. This holds moreover at the level of symplectic groupoids, and the isomorphism is given by:
\[
\xymatrixrowsep{0.4cm}
\xymatrixcolsep{1.2cm}
\xymatrix{
 (G_{\lambda}\ltimes W,\Omega_{G_{\lambda}}) \ar[r]^{\mathrm{id}\times\tau_{\lambda}}
  \ar@<-.5ex>[d] \ar@<.5ex>[d]
  & (G\ltimes {\mathfrak{g}}^*,\Omega_{G})\ar@<-.5ex>[d] \ar@<.5ex>[d] \\
 (W,\pi_{\mathfrak{g}_\lambda}) \ar[r]_{\tau_{\lambda}} & (\mathfrak{g}^*,\pi_{\mathfrak{g}})}\]
The fact that the restriction of $\Omega_{G}$ is $\Omega_{G_{\lambda}}$ can be easily checked using (\ref{eq : explicit form of Omega_G}) and (\ref{eq : conditions for isotropy}).

A direct application of Theorem \ref{thm : integrability_down} now shows that the restriction of the symplectic groupoid $G\ltimes\mathfrak{g}^*$ to the image of the map (\ref{eq : open embedding}) is isomorphic to the symplectic groupoid from e).
\end{proof}

Recall that a Lie algebra $\mathfrak{h}$ is called a {\bf Frobenius Lie algebra} if the coadjoint orbit through some $\lambda \in \mathfrak{h}^*$ is open.

\begin{illustration3}
Let $\mathfrak{h}$ be a Frobenius subalgebra of a Lie algebra $\mathfrak{g}$, and let $H \subset G$ be connected Lie groups integrating $\mathfrak{h} \subset \mathfrak{g}$. Denote by:
\[r:(\mathfrak{g}^*,\pi_{\mathfrak{g}})\rmap (\mathfrak{h}^*,\pi_{\mathfrak{h}})\]
the Poisson submersion dual to the inclusion, and by $S \subset (\mathfrak{h}^*,\pi_{\mathfrak{h}})$ the (open) symplectic leaf through $\lambda \in \mathfrak{h}^*$.
\begin{enumerate}[a)]
 \item There is an open neighborhood $\mathcal{U}$ of $0$ in $\mathfrak{h}$ so that the two-form on $\mathfrak{h}$
\[ \omega_{\lambda,x_0}(x,y)=-\lambda\left([\Xi_{x_0}x,\Xi_{x_0}y]\right),\]
is nondegenerate on $\mathcal{U}$, and the map:
\begin{equation}\label{eq : open embedding of frobenius}
(\mathcal{U},\omega_{\lambda}^{-1})\rmap (\mathfrak{h}^*,\pi_{\mathfrak{h}}), \ \ x\mapsto e^{-\mathrm{ad}_x^*}\lambda
\end{equation} is a Poisson diffeomorphism onto a neighborhood of $\lambda$ in $\mathfrak{h}^*$;
\item Around the Poisson transversal $X_{\lambda}:=r^{-1}(\lambda)$ there is a global Weinstein splitting of $\pi_{\mathfrak{g}}$ given by the commutative diagram of Poisson maps:
\begin{equation}\label{diag : infi iso}
\xymatrixrowsep{0.6cm}
\xymatrixcolsep{2.2cm}
\xymatrix{
(\mathcal{U}\times X_{\lambda},\omega_{\lambda}^{-1}+\pi_{X_{\lambda}})\ar[r]^{\phantom{12345}(x,\xi)\mapsto e^{-\mathrm{ad}_x^*}\xi}\ar[d]_{\mathrm{pr}_1}  &
(\mathfrak{g}^*,\pi_{\mathfrak{g}})\ar[d]^{r}\\
(\mathcal{U},\omega_{\lambda}^{-1})\ar[r]^{x\mapsto e^{-\mathrm{ad}_x^*}\lambda}& (\mathfrak{h}^*,\pi_{\mathfrak{h}})
}
\end{equation}
\item Theorem \ref{thm : integrability_down} for $X_{\lambda}$ implies that the restriction of the symplectic groupoid $(G\ltimes\mathfrak{g}^*,\Omega_G)$ to the image of (\ref{diag : infi iso}) is isomorphic to the product of the symplectic groupoid $(G\ltimes\mathfrak{g}^*,\Omega_G)|_{X_{\lambda}}$ and the symplectic pair groupoid $(\mathcal{U}\times\mathcal{U},\mathrm{pr}_1^*\omega_{\lambda}-\mathrm{pr}_2^*\omega_{\lambda})$;
\item The Lie-theoretic exponential of $H$, $\exp:\mathfrak{h}\to H$, induces a factorization of diagram
(\ref{diag : infi iso}) through the commutative diagram of Poisson maps:
\begin{equation}\label{diag : comm diag}
\xymatrixrowsep{0.6cm}
\xymatrixcolsep{2cm}
\xymatrix{
(H\times X_{\lambda},\mathrm{d}\widetilde{\lambda}^{-1}+\pi_{X_{\lambda}})/{H_{\lambda}}\ar[r]^{\phantom{1234567}[h,\xi]\mapsto \mathrm{Ad}_{h^{-1}}^*\xi}\ar[d]_{\mathrm{pr}_1}  &
(\mathfrak{g}^*|_S,\pi_{\mathfrak{g}})\ar[d]^{r}\\
(H,\mathrm{d}\widetilde{\lambda}^{-1})/H_{\lambda}\ar[r]^{[h]\mapsto \mathrm{Ad}_{h^{-1}}^*\lambda}& (S,\pi_{\mathfrak{h}})
}
\end{equation}
where $\widetilde{\lambda} \in \Omega^1(H)$ is the left-invariant one-form extending $\lambda$, $H_{\lambda}$ is the stabilizer of $\lambda$, and the horizontal arrows are $H$-equivariant Poisson diffeomorphisms.
\end{enumerate}
\end{illustration3}

Part d) gives a global description of the Poisson structure on the open $\mathfrak{g}^*|_S$, which implies the following:
\begin{corollary}\label{cor1}
The Poisson structure on $\mathfrak{g}^*|_S$ is horizontally nondegenerate for the submersion:
\begin{equation}\label{eq : Poisson fibration}
r:(\mathfrak{g}^*|_{S},\pi_{\mathfrak{g}})\rmap (S,\pi_{\mathfrak{h}}).
\end{equation}
The corresponding Vorobjev triple (see e.g.\ \cite{Vorobjev}) satisfies the following properties:
\begin{enumerate}[(1)]
\item The horizontal distribution is involutive, and is given by the tangent bundle to the $H$-orbits;
\item The horizontal two-form is the pullback of the symplectic form on the leaf $S$;
\item In the decomposition $\pi_{\mathfrak{g}}|_{r^{-1}(S)}=\pi^{\mathrm{v}}+\pi^{\mathrm{h}}$ into vertical and horizontal components, we have that both bivectors are Poisson and commute.
\end{enumerate}
\end{corollary}

\begin{remark}
Note that, in general, the open set $\mathfrak{g}^*|_S$ is not saturated; for example, if $\mathfrak{h}$ is the diagonal subalgebra in $\mathfrak{g}=\mathfrak{aff}(1)\oplus \mathfrak{aff}(1)$.
\end{remark}

We note also the following surprising property:

\begin{corollary}\label{cor : at most quadratic}
The induced Poisson structure $\pi_{X_{\mu}}$ on the Poisson transversal $X_{\mu}$ is at most quadratic for the canonical $\mathfrak{h}^{\circ}$-affine space structure on $X_{\mu}$.
\end{corollary}
\begin{remark}\rm
A special case of this corollary appeared in \cite{Oh} when considering the transverse Poisson structure to the coadjoint orbit through an element $\xi\in \mathfrak{g}^*$ for which the isotropy Lie algebra $\mathfrak{g}_{\xi}$ has a complement $\mathfrak{h}$ which is also a Lie algebra. In this case, note that by (\ref{eq : PT condition, linear structures 2}) $\mathfrak{h}$ is a Frobenius algebra whose orbit through $\lambda:=\xi|_{\mathfrak{h}}$ is open, and $X_{\lambda}:=\xi+\mathfrak{h}^{\circ}$ is a Poisson transversal to the coadjoint orbit of complementary dimension. Thus the corollary implies the main result of \cite{Oh}.
\end{remark}

\begin{proof}[Proof of Illustration 3]

Note that $\{\lambda\}$ is itself a Poisson transversal, with conormal bundle $\mathfrak{h}\times\{\lambda\}$, and that the pullback of $\Omega_{\mathfrak{h}}$ under $i_{\mathfrak{h}\times\{\lambda\}}:\mathfrak{h}\to \mathfrak{h}\times\{\lambda\}$, $x\mapsto (x,\lambda)$ is given by $\omega_{\lambda}$. Thus, Theorem B specializes to the diffeomorphism claimed in a).

%The fibre $X_{\mu}$ over $\mu\in \mathfrak{h}^*$ is an affine $\mathfrak{h}^{\circ}$-spaces; by Lemma \ref{lem : Poisson pulls back Poisson transversals}, fibres over points in $S$ are Poisson transversals.

The conormal bundle of $X_{\lambda}$ is $N^*X_{\lambda}=\mathfrak{h}\times X_{\lambda}$, and the relevant two-form restricted to this space, $\sigma:=-\Omega_{\mathfrak{g}}|_{\mathfrak{h}\times X_{\lambda}}$, is given by
\begin{align*}
 \sigma((x,\xi),(y,\eta))_{(x_0,\xi_0)}&=\eta(\Xi_{x_0}x)-\xi(\Xi_{x_0}y)-\xi_0([\Xi_{x_0}x,\Xi_{x_0}y])=\\
 &=-\lambda([\Xi_{x_0}x,\Xi_{x_0}y]),
\end{align*}where we have used that $TX_{\lambda}=\mathfrak{h}^{\circ}\times X_{\lambda}$, that $\xi_0|_\mathfrak{h}=\lambda$ and that $\Xi_{x_0}(\mathfrak{h})\subset \mathfrak{h}$. Hence $\sigma=\mathrm{pr}_1^*(\omega_{\lambda})$, and (\ref{diag : infi iso}) becomes Theorem \ref{thm : Normal Form for Poisson maps} for the Poisson map $r$ and the canonical sprays. This proves b) in a neighborhood of $\{0\}\times X_{\lambda}$, respectively $\{0\}$. We will conclude that b) holds on the entire $\mathcal{U}\times X_{\lambda}$ after we prove part d).

Part c) is a direct consequence of Theorem \ref{thm : integrability_down}.

The stabilizer group $H_{\lambda}$ of $\lambda$ is discrete, and therefore, the map $h\mapsto h\cdot\lambda=\mathrm{Ad}_{h^{-1}}^*\lambda$ is a local diffeomorphism from $H$ to $S$, inducing the diffeomorphism $H/H_{\lambda}\cong S$. Note also that by (\ref{eq : explicit form of Omega_G}) $-\Omega_{H}|_{H\times\lambda}=\mathrm{d}\widetilde{\lambda}$, where $\widetilde{\lambda}$ is the left invariant one-form extending $\lambda$. Therefore, restricting the right side of (\ref{diag : diagram with exps}) to $H\times\{\lambda\}$, respectively $\mathfrak{h}\times\{\lambda\}$, we obtain the following commutative diagram of local symplectomorphisms:
\begin{equation}\label{eq : triangle}
\xymatrixrowsep{0.6cm}
\xymatrixcolsep{1.2cm}
\xymatrix{
(H,\mathrm{d}\widetilde{\lambda})\ar[rd]^{h\mapsto \mathrm{Ad}_{h^{-1}}^*\lambda}  & \\
(\mathcal{O}(\mathfrak{h}),\omega_{\lambda})\ar[u]^{\exp}\ar[r]^{x\mapsto e^{-\mathrm{ad}_x^*}\lambda}& (S,\pi_{\mathfrak{h}}|_S^{-1})\\
}
\end{equation}
In particular, this shows that $\mathcal{U}\subset \mathcal{O}(\mathfrak{h})$. Also, this implies that we have an induced symplectomorphism:
\[\psi:(H,\mathrm{d}\widetilde{\lambda})/H_{\lambda}\diffto(S,\pi_{\mathfrak{h}}|_S^{-1}).\]
Since $r$ is $H$-equivariant, it follows that $\mathfrak{g}^*|_S$ is $H$-invariant. Moreover, since $S$ is the $\lambda$-orbit of $H$, it follows easily that the map $H\times X_{\lambda}\to \mathfrak{g}^*|_S$, $(h,\xi)\mapsto \mathrm{Ad}_{h^{-1}}^* \xi$ induces an $H$-equivariant diffeomorphism:
\[\Psi:H\times_{H_{\lambda}} X_{\lambda}\diffto \mathfrak{g}^*|_S,\]
which satisfies $r\circ\Psi=\psi\circ\mathrm{pr}_1$. To prove that $\Psi$ is indeed a Poisson isomorphism, note that both Poisson structures are $H$-invariant, and $\Psi$ is $H$-equivariant. Therefore, it suffices to check that $\Psi$ is a Poisson map in a neighborhood of $(H_{\lambda}\times X_{\lambda})/H_{\lambda}$, and this follows from the commutativity of  diagram (\ref{eq : triangle}), and that of diagram (\ref{diag : infi iso}) around $\{0\}\times X_{\lambda}$. On the other hand, we can now reverse the argument: having proven that (\ref{diag : comm diag}) is a commutative diagram of Poisson maps, it follows that (\ref{diag : infi iso}) is a Poisson map on the entire $\mathcal{U}\times X_{\lambda}$, respectively $\mathcal{U}$, and hence b) holds. We conclude that the factorization from d) holds:
\begin{equation*}%\label{big com diag}
\xymatrixrowsep{0.6cm}
\xymatrixcolsep{1.35cm}
\xymatrix{
(\mathcal{U}\times X_{\lambda},\omega_{\lambda}^{-1}+\pi_{X_{\lambda}})\ar[r]^{\exp\times\mathrm{Id}/H_{\lambda}\phantom{12}}\ar[d]_{\mathrm{pr}_1}&
(H\times X_{\lambda},\mathrm{d}\widetilde{\lambda}^{-1}+\pi_{X_{\lambda}})/{H_{\lambda}}\ar[r]^{\phantom{123456789}\Psi}\ar[d]_{\mathrm{pr}_1}  &
(\mathfrak{g}^*|_S,\pi_{\mathfrak{g}})\ar[d]^{r}\\
(\mathcal{U},\omega_{\lambda}^{-1})\ar[r]^{\exp/H_{\lambda}}&
(H,\mathrm{d}\widetilde{\lambda}^{-1})/H_{\lambda}\ar[r]^{\phantom{1223}\psi}& (S,\pi_{\mathfrak{h}})
}\qedhere
\end{equation*}
\end{proof}
\begin{proof}[Proof of Corollary \ref{cor1}]
By Lemma \ref{lem : Poisson pulls back Poisson transversals}, each fiber $X_{\mu}:=r^{-1}(\mu)$, $\mu\in S$, is a Poisson transversal; or equivalently, $\pi_{\mathfrak{g}}$ is horizontally nondegenerate for the map (\ref{eq : Poisson fibration}).

For $\mu\in S$, we have that $N^*X_{\mu}=\mathfrak{h}\times X_{\mu}$. Therefore, the normal bundle is given by the tangent space to the $\mathfrak{h}$-orbits:
\begin{equation}\label{eq : horizontal distribution}
N_{\xi}X_{\mu}=\pi_{\mathfrak{g}}^{\sharp}(N^*_{\xi}X_{\mu})=\pi_{\mathfrak{g}}^{\sharp}(\mathfrak{h}\times \{\xi\})=\{\mathrm{ad}_x^*\xi: x\in \mathfrak{h}\}.
\end{equation}
Since the horizontal distribution is precisely the canonical normal bundle to the fibers, this implies (1).

Diagram (\ref{diag : comm diag}) implies that the Poisson structure on $\mathfrak{g}^*|_S$ decomposes as a sum of two commuting Poisson structures $\pi_{\mathfrak{g}}|_{r^{-1}(S)}=\Psi_*(\pi_{X_{\lambda}})+ \Psi_*(\mathrm{d}\widetilde{\lambda}^{-1})$,
%\[\pi_{\mathfrak{g}}|_{r^{-1}(S)}=\pi^{\mathrm{v}}+\pi^{\mathrm{h}}, \ \ \ \pi^{\mathrm{v}}:=\Psi_*(\pi_{X_{\lambda}}), \  \ \ \pi^{\mathrm{h}}:=\Psi_*(\mathrm{d}\widetilde{\lambda}^{-1}),\]
and since $\Psi_*(\pi_{X_{\lambda}})$ is tangent to the fibres of $r$, and $\Psi_*(\mathrm{d}\widetilde{\lambda}^{-1})$ is tangent to the $H$-orbits, it follows that this is precisely the decomposition into vertical plus horizontal bivectors:
\[\pi^{\mathrm{v}}:=\Psi_*(\pi_{X_{\lambda}}), \  \ \ \pi^{\mathrm{h}}:=\Psi_*(\mathrm{d}\widetilde{\lambda}^{-1}),\]
which proves (3). Since $r$ is a Poisson map, it follows that $\pi^{\mathrm{h}}$ projects to $\pi_{\mathfrak{h}}$, and therefore, the inverse of $\pi_{\mathfrak{h}}|_S$ (i.e.\ the symplectic structure on $S$) pulls back to the inverse of $\pi^{\mathrm{h}}$ restricted to annihilator of the fibers (i.e.\ the horizontal two-form). This implies (2) (see also \cite[Proposition 3.6]{Vaisman}).
\end{proof}

\begin{proof}[Proof of Corollary \ref{cor : at most quadratic}]
By (\ref{eq : horizontal distribution}) it follows that the horizontal lift of the corresponding Ehresmann connection is given by:
\begin{equation*}
\mathrm{hor}_{\xi}:T_{\mu}S\rmap N_{\xi}X_{\mu}, \ \ \mathrm{hor}_{\xi}(\mathrm{ad}_x^*\mu)=\mathrm{ad}_x^*\xi, \ \ \xi\in X_{\mu},
\end{equation*}
and so
\[\pi_{X_{\mu},\xi}=\pi_{\mathfrak{g},\xi}-\pi^{\mathrm{h}}_{\xi}=\pi_{\mathfrak{g},\xi}-(\wedge^2\mathrm{hor}_{\xi})(\pi_{\mathfrak{h},\mu}), \ \ \xi\in X_{\mu}.\]
The claim now follows from the fact that the horizontal lift has an affine dependence on $\xi\in X_{\mu}$.
\end{proof}

\end{document}